\documentclass[12pt]{amsart}
\usepackage{amssymb}
\usepackage{graphicx,color}
\usepackage{amsmath,amscd,amsthm,amssymb,amsxtra,latexsym,epsfig,epic,graphics}
\usepackage[all]{xy}

\nonstopmode

\newtheorem{Theorem}{Theorem}[section]
\newtheorem{Lemma}[Theorem]{Lemma}
\newtheorem{Proposition}[Theorem]{Proposition}

\newtheorem{Theoremint}{Theorem}

\theoremstyle{definition}
\newtheorem{Definition}[Theorem]{Definition}
\newtheorem{Definitionint}[Theoremint]{Definition}

\newtheorem{Remark}[Theorem]{Remark}

\newtheorem{Examples}[Theorem]{Examples}
\newtheorem{Example}[Theorem]{Example}

\newcommand{\KK}{{\mathbb K}}
\newcommand{\ZZ}{{\mathbb Z}}
\newcommand{\PP}{{\mathbb P}}
\newcommand{\GG}{{\mathcal G}}

\newcommand{\HH}{{\mathcal H}}

\newcommand{\FF}{{\mathcal F}}
\newcommand{\EE}{{\mathcal E}}

\newcommand{\Ss}{{\mathcal S}}
\newcommand{\OO}{{\mathcal O}}
\newcommand{\Kk}{{\mathcal K}}

\newcommand{\Proj}{\operatorname{Proj}}
\newcommand{\sing}{\operatorname{Sing}}
\newcommand{\Supp}{\operatorname{Supp}}

\newcommand{\rk}{\operatorname{rk}}
\newcommand{\hd}{\operatorname{hd}}
\newcommand{\Hom}{\operatorname{Hom}}

\newcommand{\Ext}{\operatorname{Ext}}

\newcommand{\codim}{\operatorname{codim}}
\newcommand{\coker}{\operatorname{coker}}
\newcommand{\ann}{\operatorname{Ann}}

\makeindex

\begin{document}
\title[A Horrocks' theorem for reflexive sheaves]
{A Horrocks' theorem for reflexive sheaves}

\author[L.\ Costa, S. Marchesi, R.M.\ Mir\'o-Roig]{L.\ Costa, S. Marchesi, R.M.\
Mir\'o-Roig}

\address{Facultat de Matem\`atiques i Inform\`{a}tica,
Departament de Matem\`{a}tiques i Inform\`{a}tica, Gran Via de les Corts Catalanes
585, 08007 Barcelona, SPAIN } \email{costa@ub.edu}

\address{Instituto de Matem\'atica, Estat\'istica e Computa\c{c}\~{a}o Cient\'ifica - UNICAMP, Rua S\'ergio Buarque de Holanda 651, Distr. Bar\~ao Geraldo, CEP 13083-859, Campinas (SP), BRAZIL } \email{marchesi@ime.unicamp.br}

\address{Facultat de Matem\`atiques i Inform\`{a}tica,
Departament de Matem\`{a}tiques i Inform\`{a}tica, Gran Via de les Corts Catalanes
585, 08007 Barcelona, SPAIN } \email{miro@ub.edu}

\date{\today}


\begin{abstract} In this paper, we define  $m$-tail reflexive sheaves as reflexive sheaves on  projective spaces with the simplest possible cohomology. We prove that the rank of any $m$-tail reflexive sheaf $\EE$ on $\PP^n$ is greater or equal to $ nm-m$.  We completely describe  $m$-tail reflexive sheaves on $\PP^n$ of minimal rank and we construct huge families of $m$-tail reflexive sheaves of higher rank. \end{abstract}

\maketitle


\section{Introduction}
The study and classification of vector bundles with the simplest possible cohomology has been of great interest for many years. The first and most famous result is due to Horrocks  who proved that a vector bundle $\EE$ on $\PP^n$ without intermediate cohomology  splits into a sum of line bundles (\cite{Hor}). The next case of simplest cohomology for a vector bundle is the one when we have it all \emph{concentrated in one point}, still for the intermediate cohomology groups, i.e $H^i (\EE(\alpha)) =\KK^r$ for $i=n-1$ and a fixed integer $\alpha$ and vanishes elsewhere. Using Beilinson's spectral sequence we get that such bundles are nothing more than a direct sum of  $r$ copies of a twist of the tangent bundle $T_{\PP^n}$, plus possible direct summands of line bundles.

There have been in literature many works whose goal was the generalization of Horrocks theorem for vector bundles on projective varieties $X$ different than $\PP^n$; works that  characterize  vector bundles on $X$ either without intermediate cohomology or splitting into a direct sum of line bundles.

In this work, we want to classify reflexive sheaves with the simplest possible cohomology.
Reflexive sheaves were introduced in 1980 by Hartshorne in \cite{Har,Har2} and since then many progress have been accomplished.  According to Hartshorne's own words, the first reason for studying them is \emph{natural curiosity}; but moreover, they were defined as a new important tool for studying  rank 2 bundles on $\PP^3$ and their moduli spaces, and for the classification of space curves. Recently, the moduli problem of rank 2 reflexive sheaves has been extended to smooth projective threefolds (\cite{Ver}).

In 2008 Abe and Yoshinaga prove that  a reflexive sheaf $\FF$ on $\PP^n$ splits into a direct sum of line bundles if and only if there exists a hyperplane $H \subset \PP^n$ such that $\FF_{|H}$ also splits as a sum of line bundles (\cite{AY}; Theorem 0.2). In 2013 Yau and Ye generalize the splitting criterion to smooth projective varieties (\cite{YY}; Theorem C)  and they  give conditions to ensure that a reflexive sheaf $\EE$ on a Horrocks variety splits into a sum of line bundles and, hence, it is locally free.

The goal of our paper goes on a different direction, not considering the cases that turn out to be locally free sheaves. Indeed, we want to study  \emph{proper} reflexive sheaves, by which we mean that they are not locally free, assuming that they have the simplest possible cohomology. Therefore the first natural question that we should answer is the following one: \emph{what do we mean by the simplest possible cohomology?}

In Section \ref{sec-prel}, after recalling the necessary notation and preliminary results, we will notice that the answer will be obtained giving a closer look at the local to global spectral sequence of the $\Ext$ group. Indeed, we will ask all the intermediate cohomology to vanish except for a ``tail'' of constant dimension $m$ of the cohomology groups $H^{n-1}(\FF(\alpha))$ for sufficiently negative degrees $\alpha\in \ZZ$.
This result has motivated the following definition:

\begin{Definitionint}
Let $\FF$ be a reflexive sheaf on  $\PP^n$. We will call $\FF$ an $m$-tail reflexive sheaf if it satisfies
$$
H^i_*(\FF) = 0 \:\:\mbox{for}\:\: 1\leq i\leq n-2
$$
and
$$
h^{n-1}(\FF(t)) =
\left\{
\begin{array}{ccc}
m & \mbox{if} & t \leq k \\
0 & \mbox{if} & t > k
\end{array}
\right. \:\:\mbox{for some integer}\:\: k.
$$
\end{Definitionint}

We will see that the cohomological requests force the sheaf not to be free at a $0$-dimensional subscheme of length at most $m$.  Notice that, we could not have asked for all the intermediate cohomology to vanish, or else we will still obtain that the sheaf is a direct sum of line bundles.

After proving a lower bound for the rank of an $m$-tail reflexive sheaf on $\PP^n$, we will  focus our attention on  $m$-tail reflexive sheaves of minimal rank;  we will call them {\em minimal} $m$-tail reflexive sheaves and denote them by $\Ss_m$. We will show that  for any hyperplane $H\subset \PP^n$ avoiding the singular locus of $\Ss_m$ we have ${\Ss_m}_{|H}\cong (T_{H})^m$.
 This result is extremely useful in many ways. First of all, it tells us that the sheaves $\Ss_m$ are the clear generalization of the tangent bundles, in the sense that we have the simplest cohomology possible, beyond the one given by the sum of line bundles.

  In Section 3 we will prove the results which will give us the complete description of the minimal tail reflexive sheaves. Such description
 will be achieved proving the following steps. First we prove that every $\Ss_m$ is constructed by iterative extensions of $\Ss_1$, which we will call \emph{chain of extensions} of $\Ss_1$'s, see Proposition 3.6.
 We then prove the main structure result of the paper:
\begin{Theoremint}
Let $\Ss_m$ be a minimal $m$-tail reflexive sheaf with $s$ different singular points $p_1, \cdots,p_s$. Then,
\[\Ss_m= \oplus_{i=1}^{s} \Ss_{n_i} \]
where $\Ss_{n_i}$ is a minimal $n_i$-tail reflexive sheaf with a unique singular point $p_i$.  Moreover,  $m=n_1+\cdots+n_s$.
\end{Theoremint}
  We therefore study the matrices which give us a minimal tail reflexive sheaf singular at only one point, see Proposition 3.10 and Theorem 3.11. The last results translates our problem
  to the classification of fat points of length $m$ in the projective space.

In Section 4 we will focus on the non-minimal case. We will define a new family of $m$-tail reflexive sheaf that we will call \emph{level}, see Definition \ref{deflevel}, and we completely describe these sheaves, relating them with the minimal case.
The classification of {\em all} $m$-tail reflexive sheaves is out of reach and we will conclude the paper outlying the difficulties that we should overcome if we want to study the general case.
	\vspace{3mm}

\noindent \textbf{Acknowledgments}
All authors were partially supported by the CAPES-DGPU Proj. 7508/14-0. The first and third author are partially supported by  MTM2016-78623-P.
 The second author is partially supported by Fapesp Grant n. 2017/03487-9. The authors would like to thank the Mathematical Institute of the Universitat de Barcelona (IMUB) and the Universidade Estadual de Campinas (IMECC-UNICAMP) for the hospitality and for providing the best working conditions, and Marcos Jardim for many useful discussions.


\section{Background}\label{sec-prel}

We will work on an algebraically closed field $\KK$ of arbitrary characteristic. Given the projective space $\PP^n =\Proj(\KK[x_0,\ldots,x_n])$ and a coherent sheaf $\FF$ on $\PP^n$, we will denote the twisted sheaf $\FF \otimes \OO_{\PP^n}(l)$ by $\FF(l)$. We will write $\FF^\lor= \mathcal{H}om(\FF,\OO_{\PP^n})$ for the dual sheaf and, as usual, $H^i(\PP^n,\FF)$, or simply $H^i(\FF)$, will denote the cohomology group with dimensions $h^i(\FF)$. Given $V$ a vector space, we will denote by $V^*$ its dual. We will use the standard notation for the graded $\KK[x_0,\ldots,x_n]$-module $H^i_*(\PP^n,\FF) = \oplus_{l \in \ZZ} H^i(\PP^n,\FF(l))$.  Throughout the paper, given a matrix whose entries are linear forms, we will call \emph{change of coordinates} a finite number of elementary transformations on the rows and columns of the matrix, combined also with a change of basis on $\PP^n$.\\
We will now recall the main definitions and results used throughout the paper.

\begin{Definition}
A coherent sheaf $\FF$ on a projective variety $X$ is called reflexive if the canonical morphism $\FF \rightarrow \FF^{\lor\lor}$ is an isomorphism.
\end{Definition}
The \emph{singular locus} of a coherent sheaf $\FF$ on a smooth projective variety $X$, denoted as $\sing(\FF)$, is the set of points where $\FF$ fails to be locally free, and it is known that
$$
\begin{array}{rl}
\sing(\FF) & = \left\{ x \in X \: |\: \FF_x \:\mbox{is not a free}\: \OO_{X,x}-\mbox{module} \right\}\vspace{2mm}\\
& = \bigcup_{p=1}^{\dim X} \Supp \mathcal{E}xt^p(\FF,\OO_X)
\end{array}
$$
where $\Supp$ stands the scheme-theoretic support of the sheaf (see \cite{Ok}, Chapter 2, Lemma 1.4.1). In the particular case where $\FF$ is reflexive, we have that $\codim_X(\sing(\FF))\geq 3$ (see \cite{Har}, Corollary 1.4).
\\

From now on, we will be  interested on proper reflexive sheaves, i.e. reflexive sheaves not locally free.

\vspace{3mm}

A \emph{resolution of length $d$} of a coherent sheaf $\FF$ over  $\PP^n$ is defined as an exact sequence
$$
0\rightarrow \mathcal{L}_d \rightarrow \cdots \rightarrow \mathcal{L}_1 \rightarrow \mathcal{L}_0 \rightarrow \FF \rightarrow 0
$$
where  $\mathcal{L}_i$ splits as a direct sum of line bundles, and the minimal number of the length of such resolutions is called the \emph{homological dimension} of $\FF$ and denoted by $\hd(\FF)$.\\
Another important tool that we will use is the so called \emph{spectral sequence of global and local} $\Ext$, which states the following (see for instance \cite{W}, Section 5.8).
\begin{Theorem}
Let $\FF$, $\mathcal{G}$ be quasi coherent sheaves of $\OO_{\PP^n}$-modules. Then there is a spectral sequence with $E_2$-term $H^j(\mathcal{E}xt^i(\FF,\mathcal{G}))$ that converges to $\Ext^{i+j}(\FF,\mathcal{G})$.
\end{Theorem}

Choose $\mathcal{G}=\OO_{\PP^n}$. By looking closely at the second sheet of the latter spectral sequence, we observe that the $0$-row is given by the cohomology of the dual sheaf. We have that the simplest sheet and therefore the simplest convergence will be given if the only other non vanishing row is the one given by the cohomology of $\mathcal{E}xt^1(\FF,\OO_{\PP^n})$. \\

As noticed before, since
 $\mathcal{E}xt^{i}(\FF,\OO_{\PP^n})=0$ for $i>1$,
the support of such sheaf defines the singular locus of $\FF$  and it seems natural that the simplest case will be given when $\FF$ fails to be locally free at a finite set of points. Indeed, with the singular locus supported on points, the cohomology groups $H^j(\mathcal{E}xt^1(\FF,\OO_{\PP^n}))$ vanish for any $j>0$, simplifying even more the spectral sequence.  The only non zero value, given by $h^0(\mathcal{E}xt^1(\FF,\OO_{\PP^n}))$, will test how much the sheaf fails to be locally free. Motivated by the previous observations and recalling that, by Serre duality, $\Ext^j(\FF,\OO_{\PP^n}(-n-1))\simeq H^{n-j}(\FF)^{*}$ , we propose Definition \ref{m-tail} as the definition of proper reflexive sheaves with the simplest possible cohomology.

Specifically, we introduce a family of sheaves with no intermediate cohomology except for a ``tail'' of constant dimension $m$ for the $(n-1)$-th cohomology group. As we have just observed such sheaves represent from the cohomological point of view the simplest case of sheaves which are not locally free, because, as we will show soon (Lemma \ref{singpoint}), if all the intermediate cohomology vanishes, the sheaf is locally free and, by Horrocks, a sum of line bundles.

\begin{Definition}\label{m-tail}
Let $\FF$ be a reflexive sheaf on  $\PP^n$. We will call $\FF$ an {\em m-tail reflexive sheaf} if it satisfies
\begin{itemize}\item[(i)]
$H^i_*(\FF) = 0 \:\:\mbox{for}\:\: 1\leq i\leq n-2$, and
\item[(ii)]$
h^{n-1}(\FF(t)) =
\left\{
\begin{array}{ccc}
m & \mbox{if} & t \leq k \\
0 & \mbox{if} & t > k
\end{array}
\right. \:\:\mbox{for some integer}\:\: k.
$
\end{itemize}

We say that an $m$-tail reflexive sheaf is {\em normalized} if $k=-n-1$. Tensoring by an appropriate line bundle
, any $m$-tail reflexive sheaf can be normalized and, from now on, we will always assume that our $m$-tail reflexive sheaves are normalized.
\end{Definition}

\begin{Example}\label{examples}
(a) Any non-zero general global section  $\sigma \in H^0(T_{\PP^n}(-1))$ gives rises to an exact sequence:
$$0\rightarrow \OO_{\PP^n}(1) \stackrel{\sigma }{\rightarrow} T_{\PP^n} \rightarrow \FF \rightarrow 0$$
where $\FF$ is a 1-tail reflexive sheaf on $\PP^n$ (\cite{Ok} Example 1.1.13).

(b) Let $p_1,\cdots ,p_m$ be a set of general points in $\PP^n$ and let $\ell _i^1,\cdots ,\ell_i^n$ be $\KK$-linearly independent linear forms passing through $p_i$, $1\le i \le m$.  Set $\FF:=\coker (M^t)$ where $$
M =
\left[
\begin{array}{cccccccccccc}
M^1 & 0 & 0 & \cdots & 0\\
0 & M^2 & 0 & \ldots  & 0\\
\vdots & & \ddots & & \vdots\\
0 & 0& \cdots & 0 & M^m
\end{array}
\right]
$$
and $M^i=\left[
\begin{array}{cccccccccccc} \ell _i^1 & \cdots & \ell _i^n\end{array}
\right]
$. $\FF$ fits into a short exact sequence $$0\rightarrow \OO_{\PP^n}^m \stackrel{M^t}{\rightarrow}  \OO_{\PP^n}(1)^{nm} \rightarrow \FF \rightarrow 0$$
and we easily check that it is an $m$-tail reflexive sheaf on $\PP^n$.

(c)
Let $\FF$ be a rank  8 reflexive sheaf on $\PP^3$ defined by the following short exact sequence
$$
0\rightarrow \OO_{\PP^3}^2 \oplus \OO_{\PP^3}(2)^2 \stackrel{M^t}{\rightarrow}  \OO_{\PP^3}(1)^{6} \oplus \OO_{\PP^3}(3)^{6} \rightarrow \FF \rightarrow 0,
$$
where $M$ is given by
$$
M =
\left[
\begin{array}{ccccccccccccccccccc}
x_0 & x_1 & x_2 & x_3 & 0 & 0 & 0 & 0 & 0 &0 &0 &0\\
0 & 0 & x_0 & x_1 & x_2 & x_3 & 0 & 0 & 0 & 0 &0 &0\\
0 & 0 & 0 & 0 & 0 & 0 & x_0-x_1 & x_2 & x_3 & 0 &0 &0\\
0 & 0 & 0 & 0 & 0 & 0& 0 & 0 & 0 & x_0+x_1 & x_2 &x_3
\end{array}
\right].
$$
It is easy to check that $\FF$ is a $2$-tail reflexive sheaf on $\PP^3$. Indeed, $\FF$ is the direct sum of a rank 4 vector bundle $\FF_1$ plus two 1-tail sheaves, whose ``tail'' starts in degree -5. Using Macaulay2, see \cite{Mac2}, we check that $h^2(\FF_1(-4))=h^2(\FF_1(-3))=2$ and zero elsewhere.

(d)
Let $\FF$ be a rank 18 reflexive sheaf on $\PP^3$ defined by the following short exact sequence
$$
0\rightarrow \OO_{\PP^3}^3 \oplus \OO_{\PP^3}(2)^2 \oplus \OO_{\PP^3}(3)^3 \stackrel{M^t}{\rightarrow}  \OO_{\PP^3}(1)^{9} \oplus \OO_{\PP^3}(3)^{8} \oplus  \OO_{\PP^3}(4)^{9} \rightarrow \FF \rightarrow 0,
$$
where $M$ is given by
$$
M =
\left[
\begin{array}{c|c|c}
M_1 & 0 & 0 \\
\hline
0 & M_2 &0\\
\hline
0 & 0 & M_3
\end{array}
\right],
$$
with
$$
M_1=
\left[
\begin{array}{ccccccccccccccccccc}
x_0 & x_1 & x_2 & x_3 & 0 & 0 & 0 & 0 & 0 \\
0 & 0 & x_3 & 0 & x_0 & x_1 & x_2 & 0 & 0 \\
0 & 0 & 0 & 0 & x_3 & 0 & x_0 & x_1 & x_2
\end{array}
\right],
$$
$$
M_2=
\left[
\begin{array}{ccccccccccccccccccc}
x_0 & x_1 & x_2 & x_3 & 0 & 0 & 0 & 0\\
 0 & 0 & 0 & 0 & x_0 & x_1 & x_2 & x_3
\end{array}
\right]
$$
and
$$
M_3=
\left[
\begin{array}{ccccccccccccccccccc}
x_0 & x_1 & x_2 & 0 & 0 & 0 & 0 & 0 & 0\\
 0 & 0 & 0 & x_0 & x_1 & x_2 & 0 & 0 & 0 \\
 0 & 0 & 0 & 0 & 0& 0 & x_0 & x_1 & x_2
\end{array}
\right].
$$
It is easy to check that $\FF$ is a $3$-tail reflexive sheaf on $\PP^3$.

As in part (c), $\FF$ is the direct sum of a  bundle $\FF_1$ given by $M_1$, plus two copies of the tangent bundle given by $M_2$, plus the direct sum of three 1-tail sheaves, given by $M_3$. Using Macaulay2, we compute the cohomology of the bundle $\FF _1$ and we conclude that $\FF$ is indeed a 3-tail (decomposable) reflexive sheaf.
\end{Example}

By the cohomological properties $(i)$ and $(ii)$ of Definition \ref{m-tail},  it follows from \cite{BHH}, Proposition 1.4 that an  $m$-tail reflexive sheaf $\FF$ has homological dimension $\hd(\FF)=1$. In addition, if $\FF$ is normalized its  minimal resolution has the following shape
\begin{equation}\label{resmtail}
0 \rightarrow \bigoplus_{i=1}^s \OO_{\PP^n} (a_i) \stackrel{A}{\rightarrow} \bigoplus_{j=1}^{q} \OO_{\PP^n} (b_j) \rightarrow \FF \rightarrow 0
\end{equation}
with $a_1 \leq a_2 \leq \ldots \leq a_s$ and $b_1 \leq b_2 \leq \ldots \leq b_q$. Since $\FF$ is normalized from the fact that $h^{n-1}(\FF(-n-1))=m$ and $h^{n-1}(\FF(t))=0$ for any $t \geq -n$ we get that $a_1=0$.
Moreover if, for some $j$, $b_j \leq 0$, we have that $\OO_{\PP^n}(b_j)$ is automatically a direct summand of $\FF$, adding no information at the $(n-1)$-th cohomology group of the sheaf. Therefore, we can restrict our attention to the case $b_j \geq 1$ for any $j$, $1 \leq j \leq q$. Finally, by minimal we mean that if an entry of the matrix $A$ is given by a polynomial of degree $b_j - a_i \leq 0$, then that entry is equal to zero.

\begin{Lemma}\label{singpoint}
Let $\FF$ be an $m$-tail reflexive sheaf on $\PP^n$ with a minimal free resolution
\begin{equation}\label{seq1}
0 \rightarrow \bigoplus_{i=1}^s \OO_{\PP^n} (a_i) \stackrel{A}{\rightarrow} \bigoplus_{j=1}^{q} \OO_{\PP^n} (b_j) \rightarrow \FF \rightarrow 0.
\end{equation} Then,  $\sing(\FF)\subset \PP^n$ is a 0-dimensional scheme of length at most $m$.
In particular, if $m=0$, then $\FF$ has no intermediate cohomology and it splits as a direct sum of line bundles.
\end{Lemma}

\begin{proof}
Directly from resolution (\ref{seq1}), we obtain that $\mathcal{E}xt^p(\FF(\alpha),\OO_{\PP^n}) =0$ for each $p \geq 2$ and each  integer $\alpha$.
Applying $\mathcal{H}om(-,\OO_{\PP^n})$ to the resolution (\ref{seq1}) of the $m$-tail sheaf $\FF$, we get the exact sequence
\begin{equation}\label{dualF}
0 \rightarrow \FF^\lor \rightarrow \bigoplus_{j=1}^{q} \OO_{\PP^n} (-b_j) \rightarrow \bigoplus_{i=1}^{s} \OO_{\PP^n} (-a_i) \rightarrow \EE xt^1(\FF,\OO_{\PP^n}) \rightarrow 0
\end{equation}
and we deduce that
$$
\sing(\FF)  =  \bigcup_{p=1}^{n} \Supp \mathcal{E}xt^p(\FF,\OO_{\PP^n}) =   \Supp \mathcal{E}xt^1(\FF,\OO_{\PP^n}).
$$
Therefore the low degree terms of second sheet of the  spectral sequence of the local-global $\Ext $ (all the other ones vanish) are given by, using Serre duality in the bottom line,
$$
\begin{array}{|c|c|c|c|c|}
\hline
H^0(\mathcal{E}xt^1(\FF(\alpha),\OO_{\PP^n})) & H^1(\mathcal{E}xt^1(\FF(\alpha),\OO_{\PP^n})) & H^2(\mathcal{E}xt^1(\FF(\alpha),\OO_{\PP^n})) & \cdots & H^n(\mathcal{E}xt^1(\FF(\alpha),\OO_{\PP^n}))\\
\hline
H^0(\FF^\lor(-\alpha)) & H^1(\FF^\lor(-\alpha)) & H^2(\FF^\lor(-\alpha)) & \cdots & H^n(\FF^\lor(-\alpha))\\ \hline
\end{array}
$$
whose convergence gives the exact sequence of cohomology groups
$$
0 \rightarrow H^1(\FF^\lor(-\alpha)) \rightarrow H^{n-1}(\FF(\alpha-n-1))^* \rightarrow H^0(\mathcal{E}xt^1(\FF(\alpha),\OO_{\PP^n}))\rightarrow
$$
$$
 \rightarrow H^2(\FF^\lor(-\alpha))) \rightarrow H^{n-2}(\FF(\alpha-n-1))^*=0.
$$
Using Serre's vanishing theorem,  for $\alpha \ll 0$, $$h^0(\mathcal{E}xt^1(\FF(\alpha),\OO_{\PP^n}))=h^0(\mathcal{E}xt^1(\FF,\OO_{\PP^n})\otimes\OO_{\PP^n}(-\alpha))$$ is constantly equal to $m$, which implies that $h^0( \EE xt^1(\FF(\alpha),\OO_{\PP^n}))=m$ for every integer $\alpha$ and the support of $\mathcal{E}xt^1(\FF,\OO_{\PP^n})$ is a 0-dimensional scheme of length at most $m$.
Hence we can conclude that $\sing(\FF)$ is a  0-dimensional scheme of length at most $m$.
If $m=0$, i.e. $\FF$ does not have intermediate cohomology, we obtain that there are no points in the singular locus. Hence, $\FF$ is locally free and, by Horrocks, it splits into a direct sum of line bundles.
\end{proof}

\section{Minimal $m$-tail reflexive sheaves}

We start the section  determining the minimal rank of $m$-tail reflexive sheaves.
\begin{Lemma}\label{lemmarank}
Let $\FF$ be a reflexive $m$-tail sheaf on $\PP^n$. Then $\rk \FF \geq (n-1)m$.
\end{Lemma}
\begin{proof}
Recall that normalizing the sheaf, if necessary,   we can assume that
\begin{equation}\label{startcohom}
h^{n-1}(\FF(-n-1))=m \:\: \mbox{and} \:\: h^{n-1}(\FF(-n))=0.
\end{equation}
From the exact sequence (\ref{resmtail}) we induce the following one in cohomology
$$
0 \rightarrow H^{n-1}(\FF(-n-1)) \rightarrow \bigoplus_i H^n(\OO_{\PP^n} (a_i-n-1)) \rightarrow \bigoplus_j H^n(\OO_{\PP^n} (b_j-n-1)) = 0
$$
and the conditions (\ref{startcohom}) force to have $m$ trivial summands in the $\bigoplus_i \OO_{\PP^n} (a_i)$'s, i.e.
$$
\bigoplus_i \OO_{\PP^n} (a_i) = \OO_{\PP^n}^m \oplus \bigoplus_{i} \OO_{\PP^n} (\tilde{a}_i) \text{ with } \tilde{a}_i > 0.
$$
In the same way, from the equality $h^{n-1}(\FF(-n-2))=m$ and the minimality of (\ref{resmtail}), we must have
$$
\bigoplus_i \OO_{\PP^n} (a_i) = \OO_{\PP^n}^m \oplus \bigoplus_{i=1}^{k} \OO_{\PP^n} (\tilde{a}_i) \text{ with } \tilde{a}_i\geq 1, \text{ and }
$$
$$
\bigoplus_j \OO_{\PP^n} (b_j) = \OO_{\PP^n}(1)^{nm} \oplus \bigoplus_{j=1}^r \OO_{\PP^n} (\tilde{b}_j) \text{ with } \tilde{b}_j \geq 1.
$$
Indeed we have that, denoting by $\alpha$ the number of summands with $\tilde{a}_i=1$ and by $\beta$ the ones with $\tilde{b}_j=1$,
$$
h^{n-1}(\FF(-n-2))) - h^n(\OO_{\PP^n}(-n-2)^m) - h^n(\OO_{\PP^n}(-n-1)^\alpha) + h^n(\OO_{\PP^n}(-n-1)^\beta) - h^n(\FF(-n-2))=0,
$$
which gives us $m-m(n+1)  - \alpha + \beta \geq 0$ and hence $\beta \geq mn +\alpha \geq mn$.\\

Therefore, we obtain the following resolution of the sheaf $\FF$,
\begin{equation}\label{finalres}
0 \rightarrow \OO_{\PP^n}^m \oplus \bigoplus_{i=1}^k \OO_{\PP^n} (\tilde{a}_i) \rightarrow \OO_{\PP^n}(1)^{nm} \oplus \bigoplus_{j=1}^r \OO_{\PP^n} (\tilde{b}_j) \rightarrow \FF \rightarrow 0.
\end{equation}
Again by minimality of (\ref{finalres}),  the sheaf $\FF$ will be constructed as an extension of the sheaf defined as the cokernel of $ \OO_{\PP^n}^m {\rightarrow} \OO_{\PP^n}(1)^{nm}$.  Indeed, it is possible to fit $\FF$ in the following commutative diagram
$$
\xymatrix{
 & 0 \ar[d] & 0  \ar[d]& 0 \ar[d]& \\
 0 \ar[r] &  \bigoplus_{i=1}^k \OO_{\PP^n} (\tilde{a}_i) \ar[r]^{C^t} \ar[d] &  \bigoplus_{j=1}^q \OO_{\PP^n} (\tilde{b}_j) \ar[r] \ar[d] & \mathcal{H} \ar[r] \ar[d] & 0 \\
 0 \ar[r] & \OO_{\PP^n}^m \oplus \bigoplus_{i=1}^k \OO_{\PP^n} (\tilde{a}_i) \ar[r]^{M^t} \ar[d] & \OO_{\PP^n}(1)^{nm} \oplus \bigoplus_{j=1}^r \OO_{\PP^n} (\tilde{b}_j) \ar[r] \ar[d] & \FF \ar[r] \ar[d] & 0 \\
 0 \ar[r] & \OO_{\PP^n}^m \ar[r]^{A^t} \ar[d] & \OO_{\PP^n}(1)^{nm} \ar[r] \ar[d] & \mathcal{G} \ar[r] \ar[d] & 0\\
 & 0 & 0 & 0
}
$$
Such diagram is constructed considering its first two lines, which define commutative squares because of the degrees of the line bundles involved in the direct summands and the third line follows applying the Snake Lemma.\\
Finally, we get that $\rk \FF \geq (n-1)m$.
\end{proof}

This last Lemma motivates the following definition:

\begin{Definition}Let $\FF$ be an $m$-tail reflexive sheaf on $\PP^n$. We will say that $\FF$ is a {\em minimal} $m$-tail reflexive sheaf if it has rank  $(n-1)m$ and we will  denote it by $\Ss_m$.
\end{Definition}

\begin{Remark}\label{remarkiso}  (i) Notice that if $\FF$ is a minimal $m$-tail reflexive sheaf, by the minimality of the rank, it can not have a line bundle as a direct summand.

(ii) It is straightforward from the proof of Lemma \ref{lemmarank} that any minimal $m$-tail reflexive sheaf $\Ss_m$ has a resolution of the following type
\begin{equation}\label{resmin}
0 \rightarrow \OO_{\PP^n}^m  \stackrel{A_m^t}{\rightarrow} \OO_{\PP^n}(1)^{nm} \rightarrow \Ss_m \rightarrow 0.
\end{equation}
In the particular case $m=1$ we can suppose to have, after changing basis, the following resolution
$$
0 \rightarrow \OO_{\PP^n}  \stackrel{
{\tiny
\left[
\begin{array}{c}
x_0\\
x_1\\
\vdots\\
x_{n-1}
\end{array}
\right]}}{\rightarrow} \OO_{\PP^n}(1)^{n} \rightarrow \Ss_1 \rightarrow 0.
$$

(iii) The direct sum $\Ss _{m_1}\oplus \Ss _{m_2}$ of an $m_1$-tail reflexive sheaf and  an $m_2$-tail reflexive sheaf is an $(m_1+m_2)$-tail reflexive sheaf.
Even more, it will follow from Theorem \ref{split} that if $\Ss_m$ is  a minimal $m$-tail reflexive sheaf with $2$ different singular points $p_1$ and $p_2$ then $\Ss_m= \Ss_{m_1}\oplus \Ss_{m_2}$ with $m=m_1+m_2$,
 $\sing(\Ss_{m_1})= \{ p_1\}$ and $\sing(\Ss_{m_2})= \{ p_2\}$.

(iv) On the other hand, not all minimal $m$-tail reflexive sheaves split as a direct sum of minimal tail sheaves. For example, set $\FF:=\coker (M^t)$ where $$
M =
\left[
\begin{array}{cccccc}
x & y & z & 0 & 0 & 0 \\
t & 0 & 0 & x  & y & z\\
\end{array}
\right]
$$
 $\FF$ fits into a short exact sequence $$0\rightarrow \OO_{\PP^3}^2 \stackrel{M^t}{\rightarrow}  \OO_{\PP^3}(1)^{6} \rightarrow \FF \rightarrow 0$$
and we easily check that it is a $2$-tail reflexive sheaf on $\PP^3$ which does not split.
\end{Remark}

We will now demonstrate a result that will be extremely useful throughout this paper.
\begin{Proposition}\label{tangent}
Let $H\simeq \PP^{n-1}$ be a hyperplane of $\PP^n$ which does not meet $\sing(\Ss_m)$. Then ${\Ss_m}_{|H} \simeq T^m_{\PP^{n-1}}$, i.e., $\Ss_m$ restricts to $m$ copies of the tangent bundle on $\PP^{n-1}$.
\end{Proposition}
\begin{proof}
We can assume without loss of generality that the hyperplane is defined as $H := \{x_n=0\}$ and we consider the restriction of the resolution of $\Ss_m$ to this hyperplane
\begin{equation}\label{resrestriction}
0 \rightarrow \OO_{\PP^{n-1}}^m \stackrel{{A_m^t}_{|H}}{\rightarrow} \OO_{\PP^{n-1}}(1)^{nm} \rightarrow {\Ss_m}_{|H} \rightarrow 0.
\end{equation}
    After a possible change of basis, we can prove that the $(m \times nm)$-matrix ${A_m}_{|H}$ has a column of zeros if and only if $H^0({\Ss_m^\lor(1)}_{|H}) \neq 0$. In fact,
     as shown in \cite[Lemma 2.10]{AM}, $h^0({\Ss_m^\lor(1)}_{|H}) = p \neq 0$ if and only if ${\Ss_m(-1)}_{|H}$ splits as $\Ss' \oplus \OO_{{\PP^{n-1}}}^p$, which  implies that, after a change of basis if necessary, $p$ columns of ${{A_m}_{|H}} $ are equal to zero.  \\
Let us see that $H^0({\Ss_m^\lor(1)}_{|H}) = 0$.  To this end, considering the exact sequence
\begin{equation}\label{e1}
0 \rightarrow \Ss_m^\lor \rightarrow \Ss_m^\lor(1) \rightarrow \Ss_m^\lor(1)_{|H} \rightarrow 0
\end{equation}
it is enough to see that $H^1({\Ss_m^\lor}) = H^0({\Ss_m^\lor(1)}) = 0.$ Let us see first that $H^1({\Ss_m^\lor}) =0$.

Dualizing the resolution of $\Ss_m$ we get the following diagram
$$
\xymatrix{
0\ar[r] & \Ss_m^\lor \ar[r] & \OO_{\PP^n}(-1)^{nm} \ar[rr]^{A} \ar[dr] & & \OO_{\PP^n}^m \ar[r] & \mathcal{E}xt^1(\Ss_m,\OO_{\PP^n}) \ar[r] & 0\\
& & & K_1 \ar[ur] \ar[dr]\\
& & 0\ar[ur] & & 0
}
$$
and the following exact sequence in cohomology
$$
H^0\left(\OO_{\PP^n}(-1)^{nm}\right) \stackrel{H^0(A)}{\rightarrow} H^0\left(\OO_{\PP^n}^m\right) \rightarrow \coker H^0(A) \rightarrow 0
$$
with $\coker H^0(A) \subset H^0(\mathcal{E}xt^1(\Ss_m,\OO_{\PP^n}))$.  Since $H^0\left(\OO_{\PP^n}(-1)^{nm}\right)=0$, we have $ \dim \coker H^0(A) = m$. On the other hand, since $\Ss_m$ is an $m$-tale reflexive sheaf  $h^0(\mathcal{E}xt^1(\Ss_m,\OO_{\PP^n}))=m$. Therefore $\coker H^0(A)= H^0(\mathcal{E}xt^1(\Ss_m,\OO_{\PP^n}))$ which implies that  $h^0(K_1) =0$. Hence, from the above diagram we get $h^1(\Ss^\lor_m) =0$.\\

Now consider the dual sequence tensored by the hyperplane bundle
$$
\xymatrix{
0\ar[r] & \Ss_m^\lor(1) \ar[r] & \OO_{\PP^n}^{nm} \ar[rr]^{A(1)} \ar[dr] & & \OO_{\PP^n}(1)^m \ar[r] & \mathcal{E}xt^1(\Ss_m,\OO_{\PP^n}) \ar[r] & 0\\
& & & K_1(1) \ar[ur] \ar[dr]\\
& & 0\ar[ur] & & 0
}
$$
from where we deduce that
\begin{equation} \label{e2} H^0(\Ss_m^{\lor}(1)) \cong Ker H^0(A(1)). \end{equation}
Moreover, by Serre duality  $ Ker H^0(A(1)) \cong \coker H^n(A^t(-n-2)) $. Using the exact sequence
$$
0 \rightarrow \OO_{\PP^n}^m \rightarrow \OO_{\PP^n}^{nm}(1) \rightarrow \Ss_m \rightarrow 0
$$
together with the fact that $\Ss_m$ is an $m$-tail reflexive sheaf we deduce that $$\dim \coker H^n(A^t(-n-2))=0.$$ Therefore it follows from $(\ref{e2})$ that  $h^0(\Ss^\lor_m(1))=0$ and hence $h^0({\Ss_m^\lor(1)}_{|H})=0$.\\

Therefore, ${A_m}_{|H}$ has no column of zeros and all its entries are linear forms in  $x_0,\ldots,x_{n-1}$.  By performing a change of coordinates, we can assume that ${A_m}_{|H}$ is divided in $(1\times n)$-dimensional blocks which we denote by $A^i_m$ and whose entries are  $n$ independent linear forms $\ell_1^i,\ldots,\ell_n^i \in \KK[x_0,\ldots,x_{n-1}]$. Hence, we can describe the matrix as
$$
{A_m}_{|H} =
\left[
\begin{array}{cccccccccccc}
A^1_m & 0 & 0 & \cdots & 0\\
0 & A^2_m & 0 & \ldots  & 0\\
\vdots & & \ddots & & \vdots\\
0 & 0& \cdots & 0 & A^m_m
\end{array}
\right]
$$
and this implies that ${\Ss_m}_{|H} \simeq T^m_{\PP^{n-1}}$.
\end{proof}

\begin{Remark}
 By Proposition \ref{tangent},  $H^{n-2}({\Ss_m}_{|H}(\alpha))=0$ for any $\alpha \leq -n-1$. Hence, using the exact sequence
  \[ 0 \rightarrow \Ss_m(-1) \rightarrow \Ss_m \rightarrow {\Ss_m}_{|H} \rightarrow 0 \]
  we get that  the map $H^{n-1}(\Ss_m(\alpha-1)) \stackrel{l}{\rightarrow} H^{n-1}(\Ss_m(\alpha))$, defined by the multiplication by a linear form $l$, is an isomorphism for every $\alpha \leq -n-1$.
\end{Remark}

Our next  goal  is to describe  minimal $m$-tail reflexive sheaves on $\PP^n$. We start with  technical results, useful to better understand them.

\begin{Lemma}\label{lem3form}
Let $A_m$ be an $(m \times nm)$ matrix with linear entries defining $\Ss_m$, i.e., $\Ss_m$ is given by
$$
0 \rightarrow \OO_{\PP^n}^m  \stackrel{A_m^t}{\rightarrow} \OO_{\PP^n}(1)^{nm} \rightarrow \Ss_m \rightarrow 0.
$$
Then, after a change of basis,  $\Ss_m$ can be defined by an $(m \times nm)$ matrix with at least one row with exactly $n$ linearly independent linear forms.
\end{Lemma}
\begin{proof}
We denote by $A_m(P)$ the evaluation of the $m \times nm$ matrix $A_m$ at  a point  $P$ of $\PP^n$. Since $\Ss_m$ is not locally free,  there exists at least one point $P \in \PP^n$ such that $\rk A_m(P)<m$.
So, after a change of basis, we can assume that the first row of $A_m(P)$ is a linear combination of the other rows. Hence, if $l_1,\ldots,l_m$ are the rows of $A_m(P)$, there exist $\alpha_2,\ldots,\alpha_m \in \KK$ such that $l_1 = \alpha_2 l_2 + \ldots +\alpha_{m} l_{m}$. On the other hand, notice that if we denote by $r_1,\ldots,r_m$ the rows of $A_m$, the matrix
$$
A'_m =
\left(
\begin{array}{c}
r_1 - \alpha_2 r_2 - \ldots - \alpha_{m}r_{m} \\
\hline \vdots \\
\hline r_{m-1} \\
\hline r_m
\end{array}
\right)
$$
also defines $\Ss_m$. Moreover, the first row of $A'_m(P)$ is zero by construction, which means that all the linear forms of the first row of $A'_m$ vanish at the point $P$. Therefore, in the first row of $A'_m$   we cannot have more than $n$ linearly independent linear forms. Finally, if  in one row of the matrix defining the minimal $m$-tail reflexive sheaf $\Ss_m$ we have a number of independent linear forms strictly less than $n$, then  $\Ss_m$ would fail to be locally free on at least a line, which according to Lemma \ref{singpoint} cannot occur. Therefore, $\Ss_m$ can be defined by an $(m \times nm)$ matrix of linear entries with at least one row with exactly $n$ linearly independent linear forms.
\end{proof}

\vspace{3mm}

The following result will tell us that every minimal $m$-tail  reflexive sheaf can be obtained as a chain of extensions. We will say that $\FF$ is an {\em $m$-chain of extensions of $\Ss_1$} if it can be obtained by iterated extensions by $\Ss_1$, which means that its defining matrix is given by
$$
\left[
\begin{array}{cccccccccccc}
L_1 & 0 & 0 & \cdots & 0\\
* & L_2 & 0 & \ldots  & 0\\
\vdots & & \ddots & & \vdots\\
* & *& \cdots & * & L_m
\end{array}
\right]
$$
where the $L_i$'s are $(1\times n)$ matrices, each one defined by $n$ linearly independent linear forms.

\begin{Proposition} \label{extensions1}
Let $\Ss_m$ be a minimal $m$-tail  reflexive sheaf. Then $\Ss_m$ is an $m$-chain of extensions of $\Ss_1$'s.
\end{Proposition}
\begin{proof}
From Lemma \ref{lem3form}, we can assume that $\Ss_m$ is defined by a matrix $A_m$ of the form
$$
\left(
\begin{array}{ccc|c}
x_0 &	\cdots & x_{n-1} & 0 \cdots 0\\
\hline
& C & & B_{m-1}
\end{array}
\right).
$$
This means that $\Ss_m$ can be realized as an extension
\begin{equation}\label{extF}
0\rightarrow \Ss_1 \rightarrow \Ss_m \rightarrow \FF \rightarrow 0
\end{equation}
of $\FF$ by $\Ss_1$, obtaining the following commutative diagram
$$
\xymatrix{
 & 0 \ar[d] & 0  \ar[d]& 0 \ar[d]& \\
 0 \ar[r] & \OO_{\PP^n} \ar[r]^{[x_0 \: \cdots \: x_{n-1}]^t} \ar[d] &  \OO_{\PP^n}^n  \ar[r] \ar[d] & \Ss_1 \ar[r] \ar[d] & 0 \\
 0 \ar[r] & \OO_{\PP^n}^m \ar[r]^{A_m^t} \ar[d] & \OO_{\PP^n}(1)^{nm}  \ar[r] \ar[d] & \Ss_m \ar[r] \ar[d] & 0 \\
 0 \ar[r] & \OO_{\PP^n}^{m-1} \ar[r]^{B_{m-1}^t} \ar[d] & \OO_{\PP^n}(1)^{nm} \ar[r] \ar[d] & \mathcal{F} \ar[r] \ar[d] & 0\\
 & 0 & 0 & 0
}
$$
In fact,  the diagram is  constructed considering the first two rows and completing it using the Snake Lemma.\\

Restricting the sequence (\ref{extF}) to a hyperplane $H \cong \PP^{n-1}$ not passing through $\sing(\Ss_m) \supset \sing(\Ss_1)$, by Proposition \ref{tangent} we obtain
$$
0 \rightarrow  T_{\PP^{n-1}} \rightarrow T_{\PP^{n-1}}^m \rightarrow \FF_{|{\PP^{n-1}}} \rightarrow 0.
$$
Indeed, the restriction of the short exact sequence remains exact because, since $T_{\PP^{n-1}}$ is a simple vector bundle, i.e. $\Hom(T_{\PP^{n-1}}, T_{\PP^{n-1}}) \cong \KK$, the map $T_{\PP^{n-1}} \rightarrow T_{\PP^{n-1}}^m$ is either injective or zero. The zero case leads to contradiction, or else $\FF$ would restrict as the direct sum of $m$ copies of the tangent bundle.\\
Hence, we conclude that  $\FF_{|{\PP^{n-1}}} \cong T_{\PP^{n-1}}^{m-1}$. This also tells us that $\FF$ must not be locally free at a finite set of points, else we would not have a vector bundle as its restriction.\\
From the last row of the above diagram we have  $h^{n-1}(\FF(-n-1)) = m-1$, $h^{n-1}(\FF(t)) = 0$ for $t>-n-1$ and $H^i_*(\FF)=0$, for $i=1,\ldots,n-2$. Moreover, from
$$
0\rightarrow \FF(-1) \rightarrow \FF \rightarrow  \FF_{| \PP^{n-1}}\cong T_{\PP^{n-1}}^{m-1} \rightarrow 0
$$
together with the fact that $H^{n-2}(T_{\PP^{n-1}}(\alpha))=0$ for $\alpha \neq -n$, we have that $h^{n-1}(\FF(t))\leq m-1$ for $t \leq -n-1$.
Finally, since $\Ss_m$ and $\Ss_1$ are tail reflexive sheaves, considering (\ref{extF}) we obtain  $h^{n-1}(\FF(t))\geq m-1$ for $t \leq -n-1$. \\
Putting altogether we have proved that $\FF$ is a minimal $(m-1)$-tail reflexive sheaf on $\PP^n$.

Iterating the process, the proposition is proven.
\end{proof}

\vspace{3mm}

Now we will prove that any minimal $m$-tail reflexive sheaf $S_m$ with the singular locus $\sing(\Ss_m)$
containing at least two different points is a direct sum of minimal tail reflexive sheaves with only one point in its singular locus. This will be a consequence of the following result
\begin{Proposition}
\label{extdifpoint}
Let $\Ss_i$ be a minimal $i$-tail reflexive sheaf and $\Ss_j$ be a minimal $j$-tail reflexive sheaf such that $\sing(\Ss_i) \cap \sing(\Ss_j) = \emptyset$. Assume that $\FF$ is a minimal $(i+j)$-tail sheaf given by an extension
$$ 0\rightarrow \Ss_i \rightarrow \FF \rightarrow \Ss_{j} \rightarrow 0.$$
Then, $\FF \cong \Ss_j \oplus \Ss_i$.
\end{Proposition}
Before proving the proposition, notice that the hypothesis of $\FF$ being minimal is necessary; indeed not every extension of minimal tail is a minimal tail. Consider for example the sheaf $\FF$ defined on $\PP^3$ as the cokernel of the transposed of the matrix
$$
\left[
\begin{array}{ccccccccccccccccccccccccc}
x & y & z & 0 & 0 & 0 & 0 & 0 & 0\\
t & 0 & 0 & x & y & z &  0 & 0 & 0\\
 0 & 0 & 0 & 0 & t & 0 & x & y & z
\end{array}
\right]
$$
which is an extension of an $\Ss_2$ and $\Ss_1$ (in our notation it is also a 3-chain extension of $\Ss_1$'s).
Computing cohomology (using Macaulay2) we get that $h^2(\FF(-5)) = 3$ and $h^2(\FF(-6)) = 2$, hence $\FF$ is not a 3-tail reflexive sheaf.

\begin{proof}[Proof of Proposition \ref{extdifpoint}] Denote by $A_i$ the matrix defining $\Ss_i$ as a cokernel and by $A_j$ the matrix defining $\Ss_j$ as a cokernel.

Consider the dual of the sequence defining the extension
$$
0 \rightarrow \Ss_j^\lor \rightarrow \FF^\lor  \rightarrow \Ss_i^\lor \stackrel{f}{\rightarrow}
\mathcal{E}xt^1(\Ss_j,\OO_{\PP^n}) \rightarrow \mathcal{E}xt^1(\FF,\OO_{\PP^n})  \rightarrow \mathcal{E}xt^1(\Ss_i,\OO_{\PP^n}) \rightarrow 0.
$$
and recall we have already noticed that $h^0(\mathcal{E}xt^1(\Ss_i,\OO_{\PP^n})) = i$, $h^0(\mathcal{E}xt^1(\Ss_j,\OO_{\PP^n})) = j$ and by assumption $h^0(\mathcal{E}xt^1(\FF,\OO_{\PP^n}))=i+j$. Moreover, being the $\mathcal{E}xt$ sheaves involved supported on a 0-dimensional scheme, all their cohomology, except for their global sections, vanish. This forces the map $f$ to be zero, henceforth to the splitting of the dual sequence into the following short exact sequences
\begin{equation} \label{split}  0 \rightarrow \Ss_j^\lor \rightarrow \FF^\lor  \rightarrow \Ss_i^\lor \rightarrow 0\end{equation}
 and
\begin{equation} \label{split2}  0 \rightarrow \mathcal{E}xt^1(\Ss_j,\OO_{\PP^n}) \rightarrow \mathcal{E}xt^1(\FF,\OO_{\PP^n})  \rightarrow \mathcal{E}xt^1(\Ss_i,\OO_{\PP^n}) \rightarrow 0.  \end{equation}

Since  $\mathcal{E}xt^1(\Ss_j,\OO_{\PP^n})$ and $\mathcal{E}xt^1(\Ss_i,\OO_{\PP^n})$ are  both coherent sheaves supported on disjoint $0$-dimensional schemes, the exact sequence (\ref{split2}) splits. Hence
\[ \mathcal{E}xt^1(\FF,\OO_{\PP^n}) \cong \mathcal{E}xt^1(\Ss_i,\OO_{\PP^n}) \oplus \mathcal{E}xt^1(\Ss_j,\OO_{\PP^n}) .\]
On the other hand we have the following diagram with exact rows and columns:
$$
\xymatrix{
 &  0\ar[d]&  0\ar[d]&  0\ar[d]&   \\
 0 \ar[r] & \Ss_j^\lor \ar[d]  \ar[r] & \FF^\lor \ar[d] \ar[r] & \Ss_i^\lor \ar[d] \ar[r]& 0 \\
 0 \ar[r] & \OO_{\PP^n}(-1)^{jn} \ar[d] \ar[r]  & \OO_{\PP^n}(-1)^{(i+j)n} \ar[d] \ar[r] & \OO_{\PP^n}(-1)^{in} \ar[d] \ar[r] & 0 \\
  0 \ar[r] & \OO_{\PP^n}^{j}  \ar[d] \ar[r] & \OO_{\PP^n}^{i+j} \ar[d] \ar[r] & \OO_{\PP^n}^{i} \ar[d] \ar[r] & 0 \\
  0  \ar[r] & \mathcal{E}xt^1(\Ss_j,\OO_{\PP^n}) \ar[d] \ar[r]& \mathcal{E}xt^1(\FF,\OO_{\PP^n}) \ar[d] \ar[r ]&\mathcal{E}xt^1(\Ss_i,\OO_{\PP^n}) \ar[d] \ar[r] &0 \\
 & 0&  0& 0 &  }
 $$
From the splitting of the last row, the matrix $A$ has the two blocks $A_i^t$ and $A_j^t$ at the diagonal and zeros elsewhere. Finally cutting the second column in short exact sequences and using the fact that $\FF$ is a reflexive sheaf we get the exact sequence
\[  0  \rightarrow  \OO_{\PP^n}^{i+j}   \stackrel{A^t} \rightarrow  \OO_{\PP^n}(1)^{(i+j)n}  \rightarrow \FF \rightarrow  0.   \]
Therefore,  $\FF \cong \Ss_j \oplus \Ss_i$.
\end{proof}

Now we are ready to state our  first structure Theorem on minimal $m$-tail reflexive sheaves.  We will see that any minimal $m$-tail reflexive sheaf with different singular points splits as a direct sum of minimal tail sheaves with a unique singular point. This is a strong structural property of minimal tail reflexive sheaves that,  as we will prove in the subsequent example, does not hold for reflexive sheaves in general.

\begin{Theorem}
\label{maindiferent}
Let $\Ss_m$ be a minimal $m$-tail reflexive sheaf with $s$ different singular points $p_1, \cdots,p_s$. Then,
\[\Ss_m= \oplus_{i=1}^{s} \Ss_{n_i} \]
where $\Ss_{n_i}$ is a minimal $n_i$-tail reflexive sheaf with a unique singular point $p_i$.  Moreover,  $m=n_1+\cdots+n_s$.
\end{Theorem}
\begin{proof}
We will proceed by induction on $m$. If $m=1$ there is nothing to say. Assume $m>1$. It follows from Proposition \ref{extensions1} that $\Ss_m$ sits in an exact sequence of the following type
\begin{equation} \label{e1}  0\rightarrow \Ss_{m-1} \rightarrow \Ss_m \rightarrow \Ss_{1} \rightarrow 0 \end{equation}
where $\Ss_1$ is a minimal 1-tail sheaf singular at a only one point $p \in \sing(\Ss_m)$.

By hypothesis of induction, \[\Ss_{m-1}= \oplus_{i=1}^{w} \Ss_{\tilde{n}_i} \]
where $\Ss_{\tilde{n}_i}$ is a minimal $\tilde{n}_i$-tail reflexive sheaf with a unique singular point $p_i$, $m-1= \tilde{n}_1 + \ldots + \tilde{n}_w$ and  $w$ is either equal to $s$ or $s-1$. We will denote by $F_i$ the matrix defining $\Ss_{\tilde{n}_i}$ and by $X$ the matrix defining $\Ss_1$.

If $p \notin \sing(\Ss_{m-1})$, then $w=s-1$ and by Proposition \ref{extdifpoint} the exact sequence (\ref{e1}) splits and
$$\Ss_m \cong \Ss_{m-1} \oplus \Ss_1 \cong  \oplus_{i=1}^{s-1} \Ss_{n_i} \oplus \Ss_1 $$
and we are done.

Assume that there exist $i_0$ such that $p \in \sing(\Ss_{i_0})$. Without loss of generality we can assume that $i_0=1$. According to (\ref{e1}), the matrix associated to $\Ss_m$ is given by blocks as

$$
\left[
\begin{array}{ccccccccccccccc}
F_1 & 0 & 0 & 0 & \cdots & 0\\
0 & F_2 & 0 & 0 & \cdots & 0\\
0 &0 & F_3 & 0 & \cdots& 0\\
\vdots & \vdots& & \ddots &\ddots & \vdots\\
0 & 0 &  & & F_{w} & 0\\
A_1 & A_2 & A_3 & \cdots & A_{w} & X
\end{array}
\right]
$$

Denote by $\GG_{w}$ the reflexive sheaf defined by the matrix
$$
\left[
\begin{array}{cc}
F_{w} & 0 \\
A_{w} & X
\end{array}
\right]
$$
so that $\GG_{w}$ is given as the extension
\[ 0 \rightarrow \Ss_{n_{w}} \rightarrow \GG_{w} \rightarrow \Ss_1 \rightarrow 0. \]
Since $\Ss_m$ is a minimal tail reflexive sheaf, $\GG_{w}$ is a minimal tail reflexive sheaf. Moreover, by assumption $ p \notin \sing(\Ss_{n_{w}})$. Hence, by Proposition \ref{extdifpoint}, $\GG_{w} \cong \Ss_1 \oplus \Ss_{n_{w}}$ which implies that $A_{w}=0$.

By performing operations in rows and columns and repeating the same argument  we get that $A_2=\cdots=A_{w}=0$ and therefore $\Ss_m$ is given by the matrix

$$
\left[
\begin{array}{ccccccccccccccc}
F_{w} & 0 & 0 & 0 & \cdots & 0\\
0 & F_{s-2} & 0 & 0 & \cdots & 0\\
0 &0 & F_3 & 0 & \cdots& 0\\
\vdots & \vdots& & \ddots &\ddots & \vdots\\
0 & 0 &  & & F_{1} & 0\\
0 & 0 & 0 & \cdots & A_{1} & X
\end{array}
\right]
$$

Finally, consider $\GG_1$ given by the matrix
$$
\left[
\begin{array}{cc}
F_{1} & 0 \\
A_{1} & X
\end{array}
\right].
$$

Since $\Ss_m$ is a minimal tail sheaf, $\GG_1$ is a minimal tail sheaf with support only one point $p=p_1$ and
\[ \Ss_m \cong \oplus_{i=2}^{w} \Ss_{n_i} \oplus \GG_1\]
which proves what we want.
\end{proof}

By means of the following example we will illustrate that the property that we have just seen,  is intrinsic of tail reflexive sheaves.

\begin{Example} Let $\EE$ be a rank 2 stable reflexive sheaf on $\PP^3$  with Chern classes $(-1,4,16)$. $\EE$ admits a locally free resolution of the following type (see \cite{MR}, Theorem 2.10)
$$0\rightarrow \OO_{\PP^3}(-5)  \stackrel{(f,g,\ell)^t}{\rightarrow} \OO_{\PP^3}(-1)^2  \oplus \OO_{\PP^3}(-4)\rightarrow \EE \rightarrow 0.$$
Choosing $f,g$ general forms of degree 4 and $\ell $ a general linear form, we have by construction that $\sing(\EE)=V(f,g,\ell)$ is a set of 16 different points and since $\EE$ is stable it does not split.
\end{Example}

According to Theorem \ref{maindiferent} we can reduce the classification problem of minimal $m$-tail reflexive sheaves to the classification of minimal $m$-tail reflexive sheaves that have only one singular point.

To this end, the following is the key result

\begin{Proposition} Fix a point $p \in \PP^n$. Let $\FF$ be a rank $(n-1)m$ reflexive sheaf on $\PP^n$ with $\sing(\FF)=\{ p \}$ and given by an extension
 \begin{equation}\label{extess}
 0 \rightarrow \Ss_1^l \rightarrow \FF \rightarrow \Ss_{m-l} \rightarrow 0.
 \end{equation}
So,  $\FF$ is the cokernel of a matrix $A$ with
 \begin{equation}\label{matess}
A^t= \left[
\begin{array}{ccccccccccccccc}
X & 0 & 0 & 0 & \cdots & 0\\
0 & X & 0 & 0 & \cdots & 0\\
0 &0 & X & 0 & \cdots& 0\\
\vdots & \vdots& & \ddots &\ddots & \vdots\\
0 & 0 &  & & X & 0\\
A_1 & A_2 & A_3 & \cdots & A_{l} & B
\end{array}
\right]
\end{equation}
and $X=[\ell_1, \cdots,\ell_n]$ is given by the $n$-linearly independent forms $\ell_i$ defining $p$. Define $\GG_i$ as the cokernel of the transpose of the matrix
$$
C_i:= \left[
\begin{array}{cc}
X & 0 \\
A_{i} & B
\end{array}
\right], 1 \leq i \leq l
$$
and define $\HH_i$ as the cokernel of the transpose of the matrix obtained by deleting the $i$-th row and $i$-th block of columns of $A^t$. Then:
\begin{itemize}
\item[(a)] $\FF$ is a minimal $m$-tail reflexive sheaf if and only if $\HH_i$ is a minimal $(m-1)$-tail reflexive sheaf for $1 \leq i \leq l$.
\item[(b)] $\FF$ is a minimal $m$-tail reflexive sheaf if and only if $\GG_i$ is a minimal $(m-l+1)$-tail reflexive sheaf for $1 \leq i \leq l$.
\end{itemize}
\end{Proposition}
\begin{proof}
(a) If there exists one $\HH_i$ which is not tail, according to the exact sequence
 \[0 \rightarrow \Ss_1 \rightarrow \FF \rightarrow \HH_i \rightarrow 0 \]
 the sheaf $\FF$ would not be tail. Indeed, if $\HH_i$ is not $m-1$ tail it means that its $(n-1)$-th cohomology group decreases when we twist by subsequent negative degrees, forcing $\FF$ not to be $m$-tail. Hence we only need to prove the converse. Assume that for $1 \leq i \leq l$, $\HH_i$ is tail. Observe that for any $1 \leq i \leq l$, $\GG_i$ is tail and consider the commutative diagram
 $$
 \xymatrix{
 & & 0\ar[d] &  0 \ar[d]&  & \\
 & 0 \ar[d] \ar[r] & \Ss_1^{l-1}  \ar[d] \ar[r]& \Ss_{1}^{l-1} \ar[d] \ar[r] &  0 \\
 0 \ar[r] & \Ss_1  \ar[d] \ar[r]  & \FF \ar[d] \ar[r] & \HH_i \ar[d] \ar[r] & 0 \\
  0 \ar[r] & \Ss_1  \ar[d] \ar[r] & \GG_i \ar[d] \ar[r] & \Ss_{m-l} \ar[d] \ar[r] & 0 \\
  &  0&  0& 0 &  }
  $$
Since $\GG_i$ and $\HH_i$ are both tail reflexive sheaves, dualizing we get the following diagram:
 $$
 \xymatrix{
 & 0\ar[d] & 0\ar[d] &  0 \ar[d]&  & \\
 0 \ar[r] & \Ss_{m-l}^{\vee} \ar[d] \ar[r] & \GG_i^{\vee}   \ar[d] \ar[r]& \Ss_1^{\vee} \ar[d] \ar[r] &  0 \\
 0 \ar[r] &\HH_i^{\vee}  \ar[d] \ar[r]^{\alpha}  & \FF^{\vee} \ar[d] \ar[r] & \coker(\alpha) \ar[d] \ar[r] & 0 \\
  0 \ar[r] & {\Ss_1^{l-1}}^{\vee}  \ar[d] \ar[r]^{\simeq} & {\Ss_1^{l-1}}^{\vee} \ar[d] \ar[r] & 0 \\
  &  0&  0 &  }
  $$
By the snake lemma, $\coker(\alpha) \cong \Ss_1^{\vee}$. Therefore, since we have the short exact sequence
\[ 0 \rightarrow \HH_i^{\vee} \rightarrow \FF^{\vee} \rightarrow \Ss_1^{\vee} \rightarrow 0 \]
we get that $\FF$ is a minimal $m$-tail reflexive sheaf.

(b) Because of the definition of $\GG_i$, the result follows by applying $(l-1)$ times item $(a)$ to $\FF$, each time deleting one row and one block of columns.

\end{proof}

The previous result tells us that each $\Ss_m$, with $\sing(\Ss_m)=\{ p \}$ is obtained as the extension, described by a matrix of type (\ref{matess}) of an $\Ss_l$, also having $\sing(\Ss_l)=\{ p \}$, with $\Ss_1^{m-l}$, again having $\sing(\Ss_1)=\{ p \}$. This means that if we want to classify minimal tail reflexive sheaves, we have reduced the classification family to a simpler family of sheaves. Indeed, keeping the above notations, if $\FF$ is a minimal $m$-tail sheaf, every matrix of type $C_i$, obtained from the matrix defining $\FF$, also defines a minimal tail sheaf. Hence, if the rank of the new matrix drops by more than one when evaluate at the singular point of the associated sheaf, we can repeat the argument starting with the matrix $C_i$. Iterating the process, we will consider smaller and smaller matrices, until we arrive at a matrix $A$ with at most one row with the only non-zero coordinates given by the linear forms defining $p$.
In other words, repeating the argument we can reduce to the case of the minimal $m$-tail reflexive sheaves defined as
$$
0 \rightarrow \OO_{\PP^n}^m \stackrel{A}{\rightarrow} \OO_{\PP^n}(1)^{nm} \rightarrow \Ss_m \rightarrow 0
$$
with only one singular point $p$, such that the rank of the defining matrix $A$ drops only
by one when evaluated at the singular point $p$, i.e. $\rk A(p) = m-1$.\\

From now on we will assume that  $p=(0:\ldots:0:1)$ and we will denote by  $\EE_{\Ss_m}= \EE xt^1(\Ss_m,\OO_{\PP^n})$ so that we have the exact sequence
$$
0 \rightarrow \Ss_m^\lor \rightarrow \OO_{\PP^n}(-1)^{nm} \stackrel{A^t}{\rightarrow} \OO_{\PP^n}^m \rightarrow
\EE_{\Ss_m} \rightarrow 0.
$$

\vspace{3mm}
The next results will show us that the problem of classifying those minimal $m$-tail  reflexive sheaves with only one singular point $p$ and whose rank of the defining matrix drops only by one when evaluated at $p$ is equivalent to the problem of classifying fat points of length $m$ whose ideal has the radical associated to a simple point.

Recall the following result (see \cite{BuchEis}, Pag. 231)
\begin{Theorem}
 Let $M$ a finitely generated torsion module over a commutative ring $R$ with a free presentation
 $$
 R^p \stackrel{\varphi}{\rightarrow} R^q \rightarrow M \rightarrow 0
 $$
 then we have
 $$
 F_0(M) \subset \ann M
 $$
where $\ann M$ denotes the annihilator of the module and $F_0(M)$ the Fitting ideal defined by the maximal minors of a matrix for $\varphi$.
 \end{Theorem}
As a direct consequence of the theorem, we have that the support of $\EE_{\Ss_m}$ is a closed subscheme of the 0-dimensional scheme whose ideal
is defined by the maximal minors of the matrix $A$. Let us observe that we are sure that the last ideal defines a 0-dimensional scheme because it always contains the polynomials $x_0^m,\ldots,x_{n-1}^m$, indeed, the matrix defining the minimal tail sheaves we are focusing now can be expressed as
\vspace{3mm}
$$
 \left[
\begin{array}{ccccccccccccccc}
X & 0 & 0 & 0 & \cdots & 0\\
* & X & 0 & 0 & \cdots & 0\\
* &* & X & 0 & \cdots& 0\\
\vdots & \vdots& & \ddots &\ddots & \vdots\\
* & * &  & & X & 0\\
* & * & * & \cdots & * & X
\end{array}
\right] \:\:\:\:\mbox{with}\:\:\:\:\: X = \left[x_0\:\: x_1\:\:\ldots\:\: x_{n-1}\right].
$$

\vspace{3mm}

Therefore the support of $\EE_{\Ss_m}$ is given by the fat point $\tilde{P}$ obtained by ``adding directions'' to $P$. Because of the rank hypothesis, we have that, considered as a skyscraper sheaf on $\tilde{P}$,
$$
0 \leq \rk \EE_{\Ss_m} \leq 1.
$$
This gives us an inclusion of sheaves $\EE_{\Ss_m} \hookrightarrow \OO_{\tilde{P}}$, the last sheaf denoting the structure sheaf of the fat point. Because $h^0(\EE_{\Ss_m})=m$ by hypothesis, we conclude that $\EE_{\Ss_m}$ is the structural sheaf of a fat point of length $m$.

Let us now prove the other direction.

 \begin{Theorem}
 There is a one to one correspondence between the subschemes of $\PP^n$ of length $m$ supported at one point $p$ and the minimal $m$-tail reflexive sheaves
 $\Ss_m$ with $\sing(\Ss_m)$= \{p\} and such that the rank of its defining matrix drops by one when evaluated at $p$.
\end{Theorem}
\begin{proof}
We have already seen that the minimal $m$-tail satisfying the required rank conditions have as singular locus a fat point of length $m$.\\
 Let us now consider $\tilde{P}$ be a fat point of length $m$ whose ideal has the radical associated to one simple point $p$ (so we think of $\tilde{P}$ as $(0:\ldots:0:1)$ plus directions).\\
 Consider the structural sheaf $\OO_{\tilde{P}}$ which, using the Beilinson spectral sequence, has the following resolution
 \begin{equation}\label{res:princ}
 0 \rightarrow \OO_{\PP^{n}}(-n)^m \rightarrow \OO_{\PP^{n}}(-n+1)^{mn} \rightarrow \cdots \rightarrow \OO_{\PP^{n}}(-1)^{mn} \rightarrow \OO_{\PP^{n}}^m \rightarrow \OO_{\tilde{P}} \rightarrow 0
 \end{equation}
 We divide it  into the following exact sequences
 \begin{equation} \label{seq:split}
 \begin{array}{c}
 0 \rightarrow \Kk_1 \rightarrow \OO_{\PP^{n}}(-1)^{mn} \rightarrow \OO_{\PP^{n}}^m \rightarrow \OO_{\tilde{P}} \rightarrow 0 \vspace{0.3cm}\\
 0 \rightarrow \Kk_2 \rightarrow \OO_{\PP^{n}}(-2)^{\binom{n}{2}m} \rightarrow \Kk_1 \rightarrow 0\\
 \vdots\\
 0 \rightarrow \Kk_i \rightarrow \OO_{\PP^{n}}(-i)^{\binom{n}{i}m} \rightarrow \Kk_{i-1} \rightarrow 0\\
 \vdots\\
 0 \rightarrow \OO_{\PP^{n}}(-n)^m \rightarrow \OO_{\PP^{n}}(-n+1)^{mn} \rightarrow \Kk_{n-2} \rightarrow 0.
 \end{array}
 \end{equation}
 Dualizing the first sequence, we get
 $$
 0 \rightarrow \OO_{\PP^{n}}^m \rightarrow \OO_{\PP^{n}}(1)^{mn} \rightarrow \Kk_1^\lor \rightarrow 0.
 $$
 Our goal is to prove that $\Kk_1^\lor$ is a minimal $m$-tail reflexive sheaf. Notice that $\EE xt^1(\Kk_1^\lor,\OO_{\PP^{n}}) = \OO_{\tilde{P}}$. Therefore  from the local to global spectral sequence we get that
 $$
h^{n-1}(\Kk_1^\lor(\alpha))=m \mbox{   for   } \alpha<<0.
 $$
 By definition we also have  $h^{n-1}(\Kk_1^\lor(-n-1))=m$. Thus it is enough  to prove that the restriction of $\Kk_1^\lor$ to an hyperplane $H$ not containing the fat point is isomorphic to $m$ copies of the tangent bundle on $\PP^{n-1}$.
As in the proof of Proposition \ref{tangent},  this is true if and only if $H^0(\Kk_1(1)_{|H})=0$ and this follows observing that,  being defined in the resolution of the structure sheaf of the fat point, we have that $h^1(\Kk_1) = h^0(\Kk_1(1))=0$, which gives us the required vanishing.\\

 Henceforth, using as before the following short exact sequence
 $$
 0 \rightarrow \Kk^\lor_1(-1) \rightarrow \Kk^\lor_1 \rightarrow T_{\PP^{n-1}}^m \rightarrow 0
 $$
 we obtain that $\Kk^\lor_1 \simeq \Ss_m$, a minimal $m$-tail reflexive sheaf. Obviously, the rank of the matrix defining the sheaf drops by one when evaluated in  $p=(0:\ldots:0:1)$, or else we would not have $\EE xt^1(\Ss_m,\OO_{\PP^{n}}) = \OO_{\tilde{P}}$.
\end{proof}

Unfortunately, nowadays, this classification problem is out of range. In fact it is also related to the classification of
finite rank $n$ commutative $\KK$-algebras. It is known that if $n \leq 6$ there are finitely many types up to isomorphism, while the isomorphism
classes are of infinite number if $n \geq 7$. The complete list for rank up to 6 can be found in \cite{Poo}.
\begin{Examples}
 (a) If the fat point, in the projective space $\PP^n$, is defined by the ideal $(x_0^m,x_1,\ldots,x_{n-1})$ then the matrix $A$ of the associated
 minimal tail reflexive sheaf is the following
 $$
 A =
 \left[
 \begin{array}{cccccccccccc}
  X & 0 & \cdots & & \cdots& 0\\
  T & X & 0 & & & \vdots\\
  0 & T & X & 0 & \cdots & 0\\
  \vdots & & \ddots & \ddots & & \vdots\\
  0 &\cdots&0 & T & X & 0\\
  0 & \cdots & \cdots& 0 &  T & X
 \end{array}
 \right]
 $$
 with $X=[x_0,x_1,\ldots,x_{n-1}]$ and $T$ denotes the $1 \times n$ matrix $[x_n,0,\ldots,0]$.\vspace{0.3cm}\\
 (b) It is possible to associate to each item of Poonen's list the matrix of the associated sheaf, which will not include here for briefness.\\
 Just to give an explicit example, considering $\PP^3 = \Proj(\KK[x,y,z,t])$, take the fat point given, in the open subset defined by $\{t\neq 0\}$, by the ideal $(x^2+z^3,xy,y^2+z^3,xz,yz,z^4)$.
 The matrix of the minimal tail will be
 $$
 \left[
 \begin{array}{cccccccccccccccccccccccccccccccccccccccccccccccccccccccccccccccccccccccccccccccccccccccc}
x & y & z & 0 & 0 & 0&0&0&0&0&0&0&0&0&0&0&0&0\\
         t&0&0&x&y&z&0&0&0&0&0&0&0&0&0&0&0&0\\
         0&t&0&0&0&0&x&y&z&0&0&0&0&0&0&0&0&0\\
         0&0&t&0&0&0&0&0&0&x&y&z&0&0&0&0&0&0\\
0&0&0&0&0&0&0&0&0&0&0&t&x&y&z&0&0&0\\
0&0&0&t&0&0&0&t&0&0&0&0&0&0&t&x&y&z
\end{array}
\right]
$$
\end{Examples}


\section{Final remarks}

We end the paper with some remarks concerning non minimal $m$-tail reflexive sheaves. To do so, we start with what we call level $m$-tail reflexive sheaves.

\begin{Definition}\label{deflevel}
Let $\FF$ be an $m$-tail reflexive sheaf on $\PP^n$. We will say that $\FF$ is {\em level} if it can be
defined by a short exact sequence
\begin{equation}
0 \rightarrow \OO_{\PP^n}^m \rightarrow \bigoplus_{j=1}^q \OO_{\PP^n}(b_j) \rightarrow \FF \rightarrow 0,
\end{equation}
with $b_j \geq 1$.
\end{Definition}

For later conveniences, we will  rewrite the resolution of a level $m$-tail reflexive sheaf as
\begin{equation}\label{deflevel2}
0 \rightarrow \OO_{\PP^n}^m \rightarrow \OO_{\PP^n}(1)^{p} \oplus \bigoplus_{j=1}^{q'} \OO_{\PP^n}(b_j) \rightarrow \FF \rightarrow 0
\end{equation}
with $b_j \geq 2$. Notice that by the proof of Lemma \ref{lemmarank}, $p \geq nm$.

We first observe that any minimal $m$-tail reflexive sheaf is level and according to Example \ref{examples} (c) and (d) not all $m$-tail reflexive sheaves are level. We also notice that the following property holds:

\begin{Lemma} \label{split} Assume that $\FF$  is a normalized $m$-tail reflexive sheaf given by an extension of the following type
\[e: \quad  0 \rightarrow  \OO_{\PP^n}(a) \rightarrow \FF \rightarrow \Ss_m \rightarrow 0, \quad a \geq 0. \]
Then $\FF \cong \Ss_m \oplus \OO_{\PP^n}(a)$.
\end{Lemma}
\begin{proof} Applying the contravariant functor $\Hom(-, \OO_{\PP^n}(a))$ to the extension $e$ we get the long exact sequence:
\[ \cdots \rightarrow \Hom(\OO_{\PP^n}(a), \OO_{\PP^n}(a)) \cong \KK \stackrel{\alpha}{\rightarrow}  \Ext^1(\Ss_m, \OO_{\PP^n}(a)) \stackrel{\beta}{\rightarrow}   \Ext^1(\FF, \OO_{\PP^n}(a)) \rightarrow 0. \]
Since it is exact, $\beta \circ \alpha =0$. On the other hand $\alpha$ sends $1 \in \KK$ to the extension $e$ and $\beta$ is an isomorphism. Hence $e=0$ which implies that $\FF \cong \Ss_m \oplus \OO_{\PP^n}(a)$.
\end{proof}

In the next theorem we will see that we can reduce the description of level  $m$-tail reflexive sheaves to the description of minimal $m$-tail reflexive sheaves. Indeed, we have

\begin{Theorem}
\label{level}
Let $\FF$ be a level $m$-tail reflexive sheaf, whose resolution is described as in (\ref{deflevel2}). Then $$\FF \simeq \Ss_m \oplus \OO_{\PP^n}(1)^{p-nm} \oplus \bigoplus_{j=1}^{q'} \OO_{\PP^n}(b_j).$$
\end{Theorem}

\begin{proof}
Considering the short exact sequence (\ref{deflevel2}), we get the following commutative diagram

\begin{equation}\label{diaglevel}
\xymatrix{
& & 0 \ar[d] & 0 \ar[d] & \\
& &  \bigoplus_{j=1}^{q'} \OO_{\PP^n}(b_j) \ar[d]\ar[r]^{\simeq} & \bigoplus_{j=1}^{q'} \OO_{\PP^n}(b_j) \ar[d] & \\
0 \ar[r] & \OO_{\PP^n}^m \ar[r]^<<<<<{(A \: | \: B)^t} \ar[d]_{\simeq} & \OO_{\PP^n}(1)^{p} \oplus \bigoplus_{j=1}^{q'} \OO_{\PP^n}(b_j) \ar[r] \ar[d] & \FF \ar[r]\ar[d] & 0 \\
0 \ar[r] & \OO_{\PP^n}^m \ar[r]^{A^t}  & \OO_{\PP^n}(1)^{p}\ar[r] \ar[d]  & \mathcal{E} \ar[d] \ar[r] & 0\\
& & 0 & 0
}
\end{equation}

Indeed, we consider the long exact sequence obtained by applying the functor $\Hom(\bigoplus_{j=1}^{q'} \OO_{\PP^n}(b_j), *)$ to the exact sequence (\ref{deflevel2}) 
$$
0\rightarrow \Hom(\bigoplus_{j=1}^{q'} \OO_{\PP^n}(b_j), \OO_{\PP^n}^m ) \rightarrow \Hom(\bigoplus_{j=1}^{q'} \OO_{\PP^n}(b_j), \OO_{\PP^n}(1)^{p} \oplus \bigoplus_{j=1}^{q'} \OO_{\PP^n}(b_j)) \rightarrow \Hom(\bigoplus_{j=1}^{q'} \OO_{\PP^n}(b_j), \FF) \rightarrow \cdots
$$

Since  $b_j\geq 2$ for each $j$, $\Hom(\bigoplus_{j=1}^{q'} \OO_{\PP^n}(b_j), \OO_{\PP^n}^m )=0$ and we get the following commutative diagram
\begin{equation}\label{diaglevel}
\xymatrix{
& & 0 \ar[d] & & \\
& &  \bigoplus_{j=1}^{q'} \OO_{\PP^n}(b_j) \ar[d]\ar[r]^{\simeq} & \bigoplus_{j=1}^{q'} \OO_{\PP^n}(b_j) \ar[d]^{f} & \\
0 \ar[r] & \OO_{\PP^n}^m \ar[r]^<<<<<{(A \: | \: B)^t} \ar[d]_{\simeq} & \OO_{\PP^n}(1)^{p} \oplus \bigoplus_{j=1}^{q'} \OO_{\PP^n}(b_j) \ar[r] \ar[d] & \FF \ar[r]\ar[d] & 0 \\
& \OO_{\PP^n}^m \ar[r]^{A^t}  & \OO_{\PP^n}(1)^{p}\ar[r] \ar[d]  & \mathcal{E} \ar[d] \ar[r] & 0\\
& & 0 & 0
}
\end{equation}
which comes from the Snake Lemma applied on the two first rows in the diagram.\\
We will now show that the map $f$ is injective. Consider a non zero element $s_1$ in $H^0(\FF(-b_{q'}))$. From the fact that  $\FF$ is a reflexive sheaf and that $H^0(\FF(-b_{q'}-b))=0$ for all $b>0$, we get  the short exact sequence
\begin{equation}\label{g1}
0 \rightarrow \OO_{\PP^n}(b_{q'}) \stackrel{s_1}{\rightarrow} \FF \rightarrow \GG_1 \rightarrow 0
\end{equation}
with $\GG_1$ a torsion free sheaf.\\
Consider now an element $s_2 \in H^0(\FF(-b_{q'-1}))$, independent from $s_1$ (if $b_{q'-1}\neq b_{q'}$ this is straightforward). Using the exact sequences (\ref{deflevel2}) and (\ref{g1})  we see that $s_2$ gives us a non zero section in $H^0(\GG_1(-b_{q'-1}))$ and the following commutative diagram, where the commutativity of the square is simply given by composition,
$$
\xymatrix{
& & & 0 \ar[d]\\
& & & \OO_{\PP^n}(b_{q'-1}) \ar[d]\\
0 \ar[r] & \OO_{\PP^n}(b_{q'}) \ar[r] & \FF \ar[d]^{\simeq} \ar[r] & \GG_1 \ar[d] \ar[r] & 0 \\
& & \FF \ar[r]^{g} & \GG_2 \ar[d] \\
& & & 0
}
$$
Applying again the Snake Lemma to the two columns of the previous diagram, we obtain the short exact sequence
$$
0 \rightarrow \OO_{\PP^n}(b_{q'-1}) \rightarrow \ker g \rightarrow \OO_{\PP^n}(b_{q'}) \rightarrow 0
$$
which implies that $\ker g \simeq \OO_{\PP^n}(b_{q'-1}) \oplus \OO_{\PP^n}(b_{q'})$. Using the fact that $\FF$ is a level $m$-tail sheaf,  $H^{n-1}(\FF(-n-2))=m$ 
which implies that  $p\geq nm$ and therefore  $\rk \FF = p+q'-m \geq m(n-1) + q' > q'$. Hence, we can iterate  the process and  we obtain the required injectivity for $f$. \\

Directly from Diagram (\ref{diaglevel}) we have that
$$
h^{n-1}(\FF(-n-2)) = h^{n-1}(\EE(-n-2))=m.
$$
This implies that the matrix $A$ must have exactly $nm$ linearly independent columns, or else we will not have the right dimension on the cohomology groups required by $\FF$ being $m$-tail. This means that $\EE \simeq \EE' \oplus \OO_{\PP^n}(1)^{p-nm}$ and also $\FF \simeq \FF' \oplus \OO_{\PP^n}(1)^{p-nm}$, with $\FF'$ defined by the short exact sequence
$$
0 \rightarrow \OO_{\PP^n}^m \rightarrow \OO_{\PP^n}(1)^{nm} \oplus \bigoplus_{j=1}^{q'} \OO_{\PP^n}(b_j) \rightarrow \FF' \rightarrow 0.
$$\\
Therefore, we have that the generators of the module $H^{n-1}_*(\FF')$ are concentrated in one degree, which implies, considering that $h^{n-1}(\FF'(\alpha))=m$ for every $\alpha \leq -n-1$ and $h^{n-1}(\FF'(\alpha))=0$ for every $\alpha > -n-1$, that all maps $H^{n-1}(\FF'(\alpha-1)) \stackrel{l}{\rightarrow} H^{n-1}(\FF'(\alpha))$ are isomorphisms for every $\alpha \leq -n-1$ and the choice of a linear form $l$. Taking a hyperplane $H$ not passing through $\sing(\FF)$ and the exact sequence
$$
0 \rightarrow \FF'(\alpha-1) \rightarrow \FF'(\alpha)\rightarrow \FF'_{|H}(\alpha)\rightarrow 0
$$
we get that $H^i_*(\FF'_{|H})=0$ for $i=1,\ldots,n-3$ and $H^{n-2}_*(\FF'_{|H})=H^{n-2}(\FF'_{|H}(-n-2))$ with $h^{n-2}(\FF'_{|H}(-n-2))=m$. This implies that $\FF'_{|H} \simeq T_{\PP^{n-1}}^m \bigoplus_{j=1}^{q'} \OO_{\PP^{n-1}}(\tilde{b}_j)$. Comparing with the resolution of $\FF'$, we have  $ \bigoplus_{j=1}^{q'} \OO_{\PP^{n-1}}(\tilde{b}_j) =  \bigoplus_{j=1}^{q'}\OO_{\PP^{n-1}}(b_j)$.\\
This means that the restriction $\EE'_{|H}$ is defined as the cokernel of the map
$$
\bigoplus_{j=1}^{q'}\OO_{\PP^{n-1}}(b_j) \rightarrow T_{\PP^{n-1}}^m \oplus \bigoplus_{j=1}^{q'} \OO_{\PP^{n-1}}(b_j)
$$
hence $\EE'_{|H} \simeq  T_{\PP^{n-1}}^m \oplus \coker \beta$ with $ \bigoplus_{j=1}^{q'} \OO_{\PP^{n-1}}(b_j) \stackrel{\beta}{\rightarrow}  \bigoplus_{j=1}^{q'} \OO_{\PP^{n-1}}(b_j)$.
It follows from the commutativity of the upper right square of diagram (\ref{diaglevel}) restricted to $H$ that $\beta$ is injective.
This implies that $\beta$ is actually an isomorphism, because we could ``simplify'' the matrix that represents it by eliminating summands starting from the highest degree. Therefore,  $\EE'_{|H} \simeq  T_{\PP^{n-1}}^m$, which gives us that $\EE'$ is a minimal $m$-tail reflexive sheaf, i.e. $\EE' \simeq \Ss_m$.\\ Finally we can conclude applying Lemma \ref{split} to the last column of the diagram (\ref{diaglevel}).
\end{proof}

\begin{Remark}
Notice that in the proof of Theorem \ref{level} we have seen that the restriction of a level $m$-tail reflexive sheaf, as defined by (\ref{deflevel2}), behaves nicely. Indeed, that
  $$\FF_{|\PP^{n-1}} \cong T_{\PP^{n-1}}^m \oplus \OO_{\PP^{n-1}}(1)^{p-nm} \oplus \bigoplus_{j=1}^{q'} \OO_{\PP^{n-1}}(b_j).$$
\end{Remark}


To end the paper, we can see that in addition the general case leads to many  other different situations, involving for example the classification of Steiner bundles on the projective space, another  problem not yet solved.

Consider an $m$-tail reflexive sheaf $\FF$ on $\PP^n$, defined by the resolution
$$
0 \rightarrow \OO_{\PP^n}^m \oplus \bigoplus_{i=1}^k \OO_{\PP^n} (a_i) \stackrel{M^t}{\rightarrow} \OO_{\PP^n}(1)^{p} \oplus \bigoplus_{j=1}^q \OO_{\PP^n} (b_j) \rightarrow \FF \rightarrow 0
$$
with $a_i \geq 2 $ and $b_j\geq 2$, for each $i =1, \ldots, k$ and $j =1, \ldots, r$ and $ p \geq nm$. We can assume
$$
M = \left[
\begin{array}{c|c}
A & 0 \\
\hline
B & C
\end{array}
\right]
$$
so that we get  the following commutative diagram
$$
\xymatrix{
 & 0 \ar[d] & 0  \ar[d]& 0 \ar[d]& \\
 0 \ar[r] &  \bigoplus_{i=1}^k \OO_{\PP^n} (a_i) \ar[r]^{C^t} \ar[d] &  \bigoplus_{j=1}^q \OO_{\PP^n} (b_j) \ar[r] \ar[d] & \mathcal{H} \ar[r] \ar[d] & 0 \\
 0 \ar[r] & \OO_{\PP^n}^m \oplus \bigoplus_{i=1}^k \OO_{\PP^n} (a_i) \ar[r]^{M^t} \ar[d] & \OO_{\PP^n}(1)^{p} \oplus \bigoplus_{j=1}^q \OO_{\PP^n} (b_j) \ar[r] \ar[d] & \FF \ar[r] \ar[d] & 0 \\
 0 \ar[r] & \OO_{\PP^n}^m \ar[r]^{A^t} \ar[d] & \OO_{\PP^n}(1)^{p} \ar[r] \ar[d] & \mathcal{G} \ar[r] \ar[d] & 0\\
 & 0 & 0 & 0
}
$$
By definition, $\sing(\FF)$ consists of $m$ points.  The problem here arises from the fact that we do not know how such singular points ``distribute'' on the matrices $A$ or $C$. It could happen that $\sing(\mathcal{G})$ is empty and therefore $A$ defines a vector bundle, as in Example \ref{examples} (d), and the bundle $\mathcal{G}$ is known in literature as a \emph{Steiner bundle}. The classification of Steiner bundles on the projective space is still a question with only partial answers, so the study of $m$-tail reflexive sheaves in general will depend on the  future developments in that direction.

 \begin{Remark}
  Obviously, it is possible to classify some specific situation, for example if $m=1$ and $a_i \neq a_j$ for $i\neq j$. Indeed the only
  possibilities for $\mathcal{G}$ and $\mathcal{H}$ will be twists either of the tangent bundle on $\PP^n$ or $\Ss_1$.
 \end{Remark}


\begin{thebibliography}{999}

\bibitem{AY} T. Abe, M. Yoshinaga {\em Splitting criterion for reflexive sheaves},
Proc. Amer. Math. Soc. {\bf 136} (2008), no. 6, 1887 -- 1891

\bibitem{AM} E. Arrondo, S. Marchesi {\em Jumping pairs of Steiner bundles},
Forum Math. {\bf 27} (2015), 3233 -- 3267

\bibitem{BHH} G. Bohnhorst, H. Spindler {\em The stability of certain vector bundles on $\PP^n$},  Complex Algebraic Varieties
Proceedings of a Conference held in Bayreuth, Lecture Notes in Math, 1507 (1990).

\bibitem{BuchEis} D. Buchsbaum, D. Eisenbud, {\em What annihilates a module?},
J. of Algebra {\bf 47} (1977), 231--243

\bibitem{Eis} D. Eisenbud, {\em Commutative Algebra, with a View Toward Algebraic Geometry},
Graduate Texts in Math. {\bf 150} (2004), Springer-Verlag New York

\bibitem{Mac2} D. Grayson, M. Stillman. {\em Macaulay 2-a
software system for algebraic geometry and commutative algebra},
available at http://www.math.uiuc.edu/Macaulay2/

\bibitem{Har} R. Hartshorne, {\em Stable reflexive sheaves},
Math. Ann. {\bf 47} (1980), no. 2, 121 -- 176

\bibitem{Har2} R. Hartshorne, {\em Stable reflexive sheaves. II.},
Invent. Math. {\bf 66} (1982), no. 1, 165 -- 190

\bibitem{Hor} G. Horrocks, {\em Vector bundles on the punctured spectrum of a local ring},
Proc. London Math. Soc. {\bf 14} (1964), no. 3, 689 -- 713

\bibitem{MR} R.M. Mir\'{o}-Roig, {\em Some moduli spaces for rank 2 reflexive sheaves on $\PP^3$}, Trans.
AMS {\bf 299} (1987), 699--717.


\bibitem{Ok} C. Okonek, M. Schneider, H. Spindler {\em Vector bundles on complex projective spaces}, Progress in Math. {\bf 3} (1980), Birkh\"{a}user

\bibitem{Poo} B. Poonen {\em Isomorphism types of commutative algebras of finite rank over an algebraically closed field}
preprint available at http://www-math.mit.edu/~poonen/papers/dimension6.pdf


\bibitem{Ver} P. Vermeire {\em Moduli of reflexive sheaves on smooth projective 3-folds}, J. Pure Appl. Algebra {\bf 211} (2007), no. 3, 622 -- 632

\bibitem{W} C. Weibel {\em An introduction to homological algebra}, Cambridge studies in adv. math. {\bf 38} (1994), Cambridge University Press

\bibitem{YY} S. Yau, F. Ye {\em Several splitting criteria for vector bundles and reflexive sheaves}, Pacific Journal of Math. {\bf 266} (2013), no. 2, 449 -- 456

\end{thebibliography}
\end{document}